\documentclass[12pt]{article}
\usepackage{amsmath,amssymb,amsfonts,amsthm}
\usepackage{color}
\usepackage{hyperref}

\newtheorem{theorem}{Theorem}[section]

\newtheorem{lem}[theorem]{Lemma}

\newcommand{\1}{\partial}

\newcommand{\Z}{{\mathbb Z}}

\begin{document}
\title{Nonnegative solutions of the heat equation in a cylindrical domain and Widder's theorem}

\author{Kin Ming Hui$^*$\\
Institute of Mathematics\\ Academia Sinica\\
Taiwan, R.O.C.
\\kmhui@gate.sinica.edu.tw\\
and\\
Kai-Seng Chou\\ Institute of Mathematical Sciences
\\The Chinese University of Hong Kong\\ Hong Kong\\kschou@math.cuhk.edu.hk}

\date{Sept 6, 2023}
\maketitle
\footnote{$^*$ Corresponding author}

\begin{abstract}
It is shown every nonnegative solution of the heat equation in a bounded cylindrical domain has an integral representation in terms of a trace triple consisting of a bottom trace, a corner trace and a lateral trace on its parabolic boundary.  Conversely this trace triple uniquely determines the solution.
\end{abstract}




%

\newpage

\setcounter{equation}{0}
\setcounter{section}{0}

\section{Introduction}
\setcounter{equation}{0}
\setcounter{theorem}{0}

The study of the initial and lateral traces of nonnegative solutions of the heat equation was initiated by D.V.~Widder.  In a series of papers starting with \cite{HW}, \cite{W1}, \cite{W2}, D.V.~Widder,  P.~Hartman and A.~Winter solved the problem in the one dimensional case for solutions of the heat equation in an infinite rod, a half infinite rod and a finite rod. A complete treatment can be found in the book \cite{W3}.  In the case of an infinite rod, its $n$-dimensional version becomes the existence of the initial trace for nonnegative solutions of the heat equation in $\mathbb{R}^n$.  This problem was solved and generalized to nonnegative weak solutions of second order uniformly parabolic equations in divergence form by D.G.~Aronson in [A].

   For the case of a finite rod, in Theorem 6 of Chapter VIII of [W2] Widder showed that every nonnegative solution of the heat equation in the domain $(0,\pi)\times (0,T)$ has an integral representation in terms of three trace measures on the parabolic boundary of the domain. These traces consist of nonnegative measures $\alpha$ on $(0,\pi)\times\{0\}$ and $\beta$, $\gamma$,  on $\{0\}\times [0,T)$ and $\{\pi\}\times [0,T)$ respectively.  It was pointed out in \cite{W3} that $\alpha((\varepsilon, \pi/2)\times\{0\})$ and $\alpha((\pi/2,\pi-\varepsilon)\times\{0\})$ could be $\infty$ as $\varepsilon\to 0$ but on the other hand both $\beta(\{0\}\times (0,T_1))$ and $\gamma(\{\pi\}\times (0,T_1))$ are finite for any $0<T_1<T$.  The higher dimensional extension of a finite rod is a bounded cylindrical domain in $\mathbb{R}^n, n\geq 2$. For a non-negative solution in such domain with zero boundary data, the existence of an initial trace consisting of a bottom and a corner ones was established K.M.~Hui in  \cite{H}. 

Recently there is a lot of study of initial traces of non-negative solutions of various parabolic equations. For example similar initial trace problem for the positive solutions of the semilinear heat equation 
\begin{equation*}
u_t=\Delta u-u^q
\end{equation*}
in $C^2$ domain $\Omega\subset\mathbb{R}^n$ with compact boundary where $q>1$ was studied by M.~Marcus and L.~V\'eron \cite{MV4}. The  boundary trace problem for the corresponding elliptic problem was also 
studied by M.~Marcus and L.~V\'eron in \cite{MV1}, \cite{MV2}, \cite{MV3}. Recently K.~Hisa, K.~Ishige and J.~Takahashi \cite{HIT} proved the existence and uniqueness of initial traces of non-negative solutions  to the following semilinear heat equation,
\begin{equation*}
u_t=\Delta u+u^p
\end{equation*}
on a half space of $\mathbb{R}^N$ under the zero Dirichlet boundary condition where $p>1$. 

The initial trace problem for the porous medium equation was studied by Aronson and L.A.~Caffarelli \cite{AC}, K.S.~Chou and Y.C. Kwong \cite{CK1}, \cite{CK2}, B.E.J.~Dahlberg and C.E.~Kenig \cite{DK}, etc. The initial trace problem for the parabolic $p$-laplace equation and the doubly nonlinear parabolic equation were studied by E.~DiBenedetto and M.A.~Herrero \cite{DiH1}, \cite{DiH2}, and K.~Ishige and J.~Kinnunen \cite{Is}, \cite{IsK}, etc.

Note that in \cite{W2} only solutions of one dimensional heat equation is studied and in \cite{A} no measure initial data was considered. On the other hand in this paper we study nonnegative solutions of the heat equation in a bounded cylindrical domain in $\mathbb{R}^n$ for any $n\ge 1$ which may not be zero along their lateral boundary.  It will be shown that such solution has an integral representation in terms of a trace triple consisting of a bottom trace, a corner trace and a lateral trace on its parabolic boundary.  Conversely this trace triple uniquely determines the solution. 

To make things precise, we introduce the following notations and definitions.  Let $Q_T=\Omega\times (0,T)$ where $\Omega$ is a smooth bounded domain in $\mathbb{R}^n$ and $T>0$. For any $x\in\Omega$, let $\delta(x)=\mbox{dist}(x,\partial\Omega)$ be the distance of $x$ from the boundary $\partial\Omega$ of $\Omega$. Let $C_0^\infty(\overline{\Omega})$ be the space of all smooth functions in $\overline{\Omega}$ vanishing on $\partial\Omega$ and
\begin{equation}
L^1(\Omega,\delta)=\left\{f\in L^1_{loc}(\Omega):\int_{\Omega}|f(x)|\delta (x)\,dx<\infty\right\}.
\end{equation}
We denote by

\begin{itemize}

\item $M(\Omega,\delta)$ the collection of all nonnegative Radon measures $\mu$ on $\Omega$ satisfying
$$
\int_\Omega \delta (x)\,d\mu(x)<\infty\ ,
$$

\item $M(\partial \Omega)$ the collection of all nonnegative Radon measures on $\partial \Omega$.  Note that since every Radon measure is finite on compact sets, hence all measures in this collection are finite.

\item $M_s(\partial\Omega\times(0,T))$ the collection of all nonnegative Radon measures $\nu$ on $\partial\Omega\times(0,T))$ satisfying $\nu(\partial\Omega\times(0,T_1))<\infty$ for all $T_1\in (0,T)\ .$

\end{itemize}

Let $u$ be a classical nonnegative solution of the heat equation in $Q_T$.  We say that a pair of measures $(\mu, \lambda)\in M(\Omega, \delta)\times M(\partial\Omega)$ is the \emph{initial trace} of $u$ if for any $\varphi\in C_0^\infty(\overline{\Omega})$,
\begin{equation*}
\lim_{t\to 0^+}\int_\Omega \varphi(x) u(x,t)\,dx =\int_\Omega \varphi \,d\mu +\int_{\partial\Omega}\dfrac{\partial\varphi}{\partial N}\,d\lambda\ ,
\end{equation*}
where $\partial/\partial N$ is the derivative with respect to the unit inner normal $N$. On the other hand we say that a measure $\nu\in M_s(\partial\Omega\times(0,T))$ is the \emph{lateral trace} of $u$ if for any $h\in C_c(\partial\Omega\times(0,T))$,
 \begin{equation*}
 \lim_{\varepsilon\to 0}\int_0^T\int_{\partial\Omega_\varepsilon} \tilde{h}(x,t)u(x,t)\,d\sigma(x) dt =\iint_{\partial\Omega\times(0,T)}h\,d\nu\ ,
 \end{equation*}
where $\Omega_\varepsilon=\{x\in\Omega:\ \delta(x)>\varepsilon\}$ and $\tilde{h}$ is any continuous extension of $h$ in a tubular neighborhood of $\1\Omega\times (0,T)$ vanishing near $t=0, T$.

Our main result is

\begin{theorem}\label{representation-thm}
Let $u$ be a nonnegative classical solution of the heat equation in $Q_T$.  There exists $(\mu,\lambda, \nu)\in M(\Omega, \delta)\times M(\partial\Omega)\times M_s(\partial\Omega\times(0,T))$ such that $(\mu,\lambda)$ is the initial trace and $\nu$ is the lateral trace for $u$.  Moreover,
\begin{eqnarray}\label{representation-formula}
u(x,t)&=&\int_\Omega G(x,t;y,0)\,d\mu(y)+\int_{\partial\Omega}\dfrac{\partial G}{\partial N_y}(x,t;y,0)\,d\lambda(y) \nonumber\\
 &  &  +\iint_{\partial\Omega\times(0,t)}\dfrac{\partial G}{\partial N_y}(x,t;y,s)\,d\nu(y,s)\quad\forall (x,t)\in Q_T,
\end{eqnarray}
where $\partial/\partial N_y$ is the derivative with respect to the unit inner normal $N_y$  at $y\in \partial\Omega$ and $G(x,t;y,s)$ is the Green kernel of the heat equation in $\Omega\times \mathbb{R}$.

Conversely, given any $(\mu,\lambda,\nu)\in M(\Omega, \delta)\times M(\partial\Omega)\times M_s(\partial\Omega\times(0,T))$,  \eqref{representation-formula} gives a classical solution of the heat equation in $Q_T$ whose trace triple is equal to  $(\mu,\lambda,\nu)$.

\end{theorem}

 Along a different vein, integral representation formulas for nonnegative solutions of the heat equation in bounded domains in term of the so-called kernel function were given by J.T.~Kemper in [K].  His results were extended to nonnegative solutions of uniformly parabolic divergence equations  by E.B.~Fabes, N.~Garofalo and S.~Salsa in [FGS].  One may consult [M] by M.~Murata for more recent results in this direction.

For the related generalized porous medium equation, existence and uniqueness of a pair of initial traces for its nonnegative solution of
the initial Dirichlet problem in a bounded smooth cylindrical domain was proved by B.E.J.~Dahlberg and C.E.~Kenig in \cite{DK}.  Results on non-zero lateral traces for solutions of the porous medium equation were obtained by K.S.~Chou and Y.C.~Kwong in [CK1] and [CK2].

The plan of the paper is as follows. In section 2 we will prove the existence of initial trace for nonnegative solution of the heat equation in bounded smooth cylindrical domain. In section 3 we will prove the existence of the lateral trace and the representation formula \eqref{representation-formula} for the solution. We will also prove Theorem \ref{representation-thm} in section 3.

\section{Existence of initial trace}
\setcounter{equation}{0}
\setcounter{theorem}{0}

We will assume that $\Omega$ is a smooth bounded domain  in $\mathbb{R}^n$ and $u$ is a nonnegative solution of the heat equation in the cylindrical domain $Q_T$ for some constant $T>0$ for the rest of the paper.
We will fix a sufficiently small $\varepsilon_0>0$ such that the boundary of the subdomain
\begin{equation*}
 \Omega_{\varepsilon}=\{x\in\Omega:\mbox{ dist}\,(x,\partial\Omega)>\varepsilon\},\qquad \varepsilon\in [0,\varepsilon_0]
\end{equation*}
is smooth and for each $x\in\Omega\setminus \Omega_{\varepsilon_0}$, there is a unique $z=z(x)\in \partial\Omega$ satisfying $x=z+\delta(x)N_z$ where  $\delta(x)$ is the distance from $x$ to $\partial\Omega$ and $N_z$ is the unit inner normal of $\partial\Omega$ at $z$. The map $x\mapsto (z,\delta(x))\in \partial \Omega\times(0,\varepsilon_0]$ forms a diffeomorphism from $\Omega\setminus\Omega_{\varepsilon_0}$ to $\partial\Omega\times (0,\varepsilon_0]$.  We will extend the distance function $\delta$ on $\Omega\setminus\Omega_{\varepsilon_0}$  to a smooth positive function $\overline{\delta}$ on $\Omega$ and fix it throughout this paper.  Apparently such choice of $\overline{\delta}$ does not alter $M(\Omega,\delta)$.

We first recall some basic properties of the Green kernel $G(x,t;y,s)$ of  the heat equation (cf. \cite{C}, \cite{I}).  Note that the Green kernel $G(x,t;y,s)$ for the heat equation in $\Omega\times\mathbb{R}$ exists and is a continuous function in $$
\{(x,t,y,s):\ x,y\in\overline{\Omega}, -\infty<s<t<\infty\}
$$
which is smooth in its interior such that

\begin{itemize}

\item for each $(y,s)\in \Omega\times \mathbb{R}, G(\cdot, \cdot; y,s)>0$ satisfies the heat equation in $\Omega\times (s,\infty)$, vanishes on $\partial\Omega\times(s, \infty)$ and satisfies
\begin{equation*}
\lim_{t\searrow s}G(x,t; y,s)=\delta_y\quad\forall s\in\mathbb{R}
\end{equation*}
in the distribution sense where $\delta_y$ is the delta mass at $y$.

\item for each $(x,t)\in \Omega\times\mathbb{R},  G(x,t; \cdot,\cdot)$ satisfies the backward heat equation in $\Omega\times(-\infty, t)$ and $G(x,t; y,s)=0$, $\frac{\partial G_\varepsilon}{\partial N_y} (x,t;y,s)>0$ for all $(y,s)\in \partial\Omega\times(-\infty, t)$.

\item and if we let $\Omega_1\subset \Omega_2$ and $G_1$, $G_2$, be the Green kernel of the heat equation with respect to the cylindrical domains $\Omega_1\times\mathbb{R}$ and $\Omega_2\times \mathbb{R}$ respectively, then \begin{equation*}
G_1(x,t; y,s)\leq G_2(x,t; y,s)\quad\forall x,y\in \Omega_1, s<t.
\end{equation*}
\end{itemize}

The proof of our main theorem Theorem \ref{representation-thm} will be accomplished in several lemmas. First of all, by the integral representation formula for solutions of the heat equation in cylindrical domain (Theorem 5 of Chapter VII of  \cite{C}), $u$ admits the following integral representation, namely, for any $(x,t)\in\Omega_\varepsilon\times (s,T), \varepsilon\in (0,\varepsilon_0)$,

\begin{equation}\label{representation1}
u(x,t)=\int_{\Omega_\varepsilon}G_{\varepsilon}(x,t;y,s)u(y,s)\,dy+\int_s^t\int_{\partial\Omega_\varepsilon}\dfrac{\partial G_{\varepsilon}}{\partial N_y}(x,t;y,\tau)u(y,\tau)\,d\sigma(y)\,d\tau\ ,
\end{equation}
where $G_\varepsilon$ is the Green kernel for the heat equation in $\Omega_{\varepsilon}\times \mathbb{R}$ and $\partial/\partial N_y$ is the derivative with respect to the unit inner normal $N_y$ at $y\in\partial\Omega_\varepsilon$.  Note that both $G_\varepsilon(x,t;y,s)$ and $\frac{\partial G_\varepsilon}{\partial N_y} (x,t;y,s)$ are positive.

Since both terms on the right hand side of \eqref{representation1} are nonnegative, we have, for all $x\in\Omega_\varepsilon,\ 0<s<t<T, \ 0<\varepsilon\le\varepsilon_0,$

\begin{equation}\label{bottom-integral-upper-bd}
\int_{\Omega_\varepsilon}G_{\varepsilon}(x,t;y,s)u(y,s)\,dy\le u(x,t)\ ,
\end{equation}

and

\begin{equation}\label{lateral-integral}
\int_s^t\int_{\partial\Omega_\varepsilon}\dfrac{\partial G_{\varepsilon}}{\partial N_y}(x,t;y,\tau)u(y,\tau)\,d\sigma(y)\,d\tau\le u(x,t)\ .
\end{equation}
Since $G_\varepsilon(x,t; \cdot, \cdot)\uparrow G(x,t; \cdot,\cdot)$ as $\varepsilon\to 0$, letting $\varepsilon\to 0$ in \eqref{bottom-integral-upper-bd}, by the monotone convergence theorem,
we have
\begin{equation}\label{initial-value-integral}
\int_{\Omega}G(x,t;y,s)u(y,s)\,dy\le u(x,t) \quad\forall x\in\Omega,\ 0<s<t<T\ .
\end{equation}

\begin{lem}\label{u-weighted-l1-integral-uniform-bd-lem}
For  any $T_1\in (0,T)$, we have
\begin{equation*}
\sup_{0<t\le T_1}\int_{\Omega} u(x,t)\delta (x)\,dx<\infty.
\end{equation*}
\end{lem}
\begin{proof}
 We fix some $x_0\in \Omega$ and $T_2\in (T_1,T)$. By \eqref{initial-value-integral}, we have
\begin{equation}\label{green-kernel-integral-upper-bd}
\int_{\Omega} G(x_0,T_2;w,s)u(w,s)\,dw \leq u(x_0,T_2) \quad\forall\ 0<s<T_2.
\end{equation}
Now for any $(w,s)\in (\Omega\setminus\Omega_{\varepsilon_0})\times(0,T_1]$, we have
\begin{align}\label{green-kernel-near-bdary}
&G(x_0,T_2;w,s)\notag\\
=&G(x_0,T_2;z(w)+\delta(w)N_{z(w)},s)-G(x_0,T_2;z(w),s)\notag\\
=&\int_0^1 \dfrac{\partial G}{\partial a}(x_0,T_2;z(w)+a\delta(w)N_{z(w)},s)\,da\notag\\
=&\left(\sum_{j=1}^nN_j(z(w))\int_0^1 \dfrac{\partial G}{\partial y_j}(x_0,T_2;z(w)+a\delta(w)N_{z(w)},s)\,da\right) \delta(w)
\end{align}
where $N_z=(N_1(z),\cdots,N_n(z))$ is the unit inner normal at $z\in\partial\Omega$.
Since $\partial G/\partial N_y(x_0,T_2;y,s)$ is positive for $y\in\partial\Omega$ and $s<T_2$ and also uniformly continuous for $s\in (0,T_1]$, there exist constants $0<\varepsilon_1<\varepsilon_0$ and $c_1>0$ such that
\begin{equation}\label{green-kernel-derivative-near-bdary}
c_1\leq \sum_{j=1}^nN_j(z(w))\dfrac{\partial G}{\partial y_j}(x_0,T_2;z(w)+a\delta(w)N_{z(w)},s)\leq \frac{1}{c_1}
\end{equation}
holds for any $(w,s)\in (\Omega\setminus\Omega_{\varepsilon_1})\times(0,T_1]$.
Therefore by \eqref{green-kernel-near-bdary} and \eqref{green-kernel-derivative-near-bdary}, we have
\begin{equation}\label{green-kernel-distance-fcn-equivalence}
c_1\delta(w)\leq G(x_0,T_2;w,s)\leq c_1^{-1}\delta(w)\quad\forall y\in\Omega\setminus\Omega_{\varepsilon_1},\ 0<s\le T_1.
\end{equation}
Since both $\delta(w)$ and $G(x_0,T_2;w,s)$ are  positive and uniformly bounded above and below by some positive constants in $\overline{\Omega}_{\varepsilon_1}$, by \eqref{green-kernel-integral-upper-bd} and \eqref{green-kernel-distance-fcn-equivalence}
the lemma follows.
\end{proof}

\begin{lem}\label{lateral-u-integral-bd-lem}
For any $T_1\in (0,T)$, we have
\begin{equation*}
\sup_{0<\varepsilon\le\varepsilon_0}\int_0^{T_1}\int_{\partial\Omega_{\varepsilon}}u(x,\tau)\,d\sigma (x)\,d\tau<\infty.
\end{equation*}
\end{lem}
\begin{proof}
Let $x_0\in\Omega_{\varepsilon_0}$, $T_2\in (T_1,T)$ and $0<\varepsilon\le\varepsilon_0$. Putting $x=x_0$ and $t=T_2$ in \eqref{lateral-integral} and letting $s\to 0$,
by the monotone convergence theorem, we have

\begin{equation}\label{green-derivative-lateral-integral}
\int_0^{T_1}\int_{\partial\Omega_{\varepsilon}}\dfrac{\partial G_{\varepsilon}}{\partial N_y}(x_0,T_2;y,\tau)u(y,\tau)\,d\sigma(y)\,d\tau\le u(x_0,T_2)\ .
\end{equation}
We now claim that there exists a constant $c_1>0$ such that

\begin{equation*}
\dfrac{\partial G_{\varepsilon}}{\partial N_y}(x_0,T_2;y,\tau)\ge c_1\ ,\quad\forall y\in\partial\Omega_{\varepsilon},\ 0\le\tau\le T_1,\ 0<\varepsilon\le\varepsilon_0.
\end{equation*}
Suppose the claim does not hold. Then there exist sequences $\{\varepsilon_i\}_{i=1}^{\infty}\subset (0,\varepsilon_0]$, $\{\tau_i\}_{i=1}^{\infty}\subset [0,T_1]$, $\{y_i\}_{i=1}^{\infty}$, such that $y_i\in\partial\Omega_{\varepsilon_i}$ for any $i\in\Z^+$ and
\begin{equation*}
\dfrac{\partial G_{\varepsilon_i}}{\partial N_{y_i}}(x_0,T_2;y_i,\tau_i)\to 0\ ,\quad\mbox{ as }i\to\infty.
\end{equation*}
Then there exists $\overline{\varepsilon}_0\in [0,\varepsilon_0]$, $y_0\in\partial{\Omega}_{\overline{\varepsilon}_0}$, $\tau_0\in [0,T_1]$ and  subsequences of $\{\varepsilon_i\}$, $\{y_j\}$, $\{\tau_i\}$, which we may assume without loss of generality to be the sequences themselves such that $\varepsilon_i\to \overline{\varepsilon}_0$, $y_j\to y_0$ and
$\tau_i\to\tau_0$ as $i\to\infty$. It follows that
\begin{equation*}
\dfrac{\partial G_{\overline{\varepsilon}_0}}{\partial N_{y_0}}(x_0,T_2;y_0,\tau_0)=0.
\end{equation*}
On the other hand since $\partial G_{\overline{\varepsilon}_0}/\partial N_{y_0}(x_0,T_2;y,\tau)$ is positive for any $y\in\partial{\Omega}_{\overline{\varepsilon}_0}$ and $\tau\in [0,T_1]$, contradiction arises and our claim holds.
By \eqref{green-derivative-lateral-integral} and the claim, the lemma follows.

\end{proof}

\begin{lem}
For any $T_1\in (0,T)$, we have
$$
\int_0^{T_1}\int_{\Omega} u(x,t)\,dx\, dt<\infty.
$$
\end{lem}
\begin{proof}
For any $0<\varepsilon<\varepsilon_0$, let $\varphi_\varepsilon$ be the solution of
\begin{equation}\label{phi-eqn}
\left\{\begin{aligned}
-\Delta \varphi&=1\quad\mbox{ in }\Omega_\varepsilon\\
\varphi(x)&=0\quad\forall x\in \partial\Omega_\varepsilon.
\end{aligned}\right.
\end{equation}
 According to elliptic theory (cf. \cite{GT}, \cite{Wi}), by decreasing $\varepsilon_0$   if necessary, there exists a constant $C_2>0$ such that
\begin{equation}\label{green-fcn-derivative-upper-lower-bd}
C_2\leq \frac{\partial\varphi_\varepsilon}{\partial N_y}(y)\leq C_2^{-1}\quad\forall y\in \partial\Omega_{\varepsilon}, \varepsilon\in (0,\varepsilon_0).
\end{equation}
  Moreover
\begin{equation}\label{phi-delta-ineqn}
\varphi_\varepsilon (x)\leq C\delta_{\varepsilon}(x)\leq C\delta(x)\quad\forall x\in\Omega_\varepsilon
\end{equation}
for some constant $C>0$ where $\delta_{\varepsilon}(x)=\mbox{dist}(x,\partial\Omega_{\varepsilon})$. Multiplying the heat equation by $\varphi_{\varepsilon}$ and integrating over
$\Omega_{\varepsilon}\times (t,T_1)$, we have

\begin{align}\label{phi-epsilon-integral-eqn}
&\int_{\Omega_\varepsilon}u(x, T_1)\varphi_{\varepsilon}(x)\,dx -\int_{\Omega_\varepsilon}u(x,t)\varphi_{\varepsilon}(x)\,dx\notag\\
=&\int_t^{T_1}\int_{\Omega_\varepsilon}u\Delta \varphi_{\varepsilon} \,dx\,dt +\int_t^{T_1}\int_{\partial\Omega_\varepsilon}u\dfrac{\partial\varphi_{\varepsilon}}{\partial N}\,d\sigma
\, dt.
\end{align}
By \eqref{phi-eqn}, \eqref{green-fcn-derivative-upper-lower-bd}, \eqref{phi-delta-ineqn}, \eqref{phi-epsilon-integral-eqn}  and Lemma \ref{u-weighted-l1-integral-uniform-bd-lem}, we have
\begin{equation}\label{u-x-t-integral-ineqn}
C_2\int_t^{T_1}\int_{\partial\Omega_\varepsilon}u(x,t)\,d\sigma (x)\,dt
\leq C_3+\int_t^{T_1}\int_{\Omega_\varepsilon}u(x,t)\,dx\,dt
\end{equation}
for some constant $C_3>0$ independent of $\varepsilon\in (0,\varepsilon_0)$. Now using the coarea formula \cite{EG}, we have
$$
 -\dfrac{d }{d\varepsilon}\int_{\Omega_\varepsilon}g\,dx =\int_{\partial \Omega_\varepsilon}\dfrac{g}{|\nabla \delta (x)|}\,d\sigma (x)
$$
where
$$
g(x)=\int_{t}^{T_1} u(x,\tau)\,d\tau
$$
and noting that
\begin{equation*}
|\nabla \delta (x)|\geq C_4\quad\forall x\in\Omega\setminus\Omega_{\varepsilon_0}
\end{equation*}
for some constant $C_4>0$, by \eqref{u-x-t-integral-ineqn} we get
\begin{equation}\label{g-epsilon-ineqn}
-\dfrac{d }{d\varepsilon}G(\varepsilon) \leq C_5G(\varepsilon) + C_5
\end{equation}
for some constant $C_5>0$ independent of $\varepsilon\in (0,\varepsilon_0)$ where
$$
G(\varepsilon)=\int_{t}^T\int_{\Omega_\varepsilon}u\,dx\, dt.
$$
Integrating \eqref{g-epsilon-ineqn} from $\varepsilon$ to $\varepsilon_0$, we have
\begin{equation}\label{u-integral-ineqn2}
\int_t^T\int_{\Omega_\varepsilon}u\,dx\,dt \leq e^{C_5(\varepsilon_0-\varepsilon)}\int_t^T\int_{\Omega_{\varepsilon_0}}u\,dx\,dt +e^{C_5(\varepsilon_0-\varepsilon)}.
\end{equation}
Letting first $\varepsilon\to 0$ and then $t\to 0$ in \eqref{u-integral-ineqn2}, the lemma follows.
\end{proof}

\begin{lem}\label{existence-bdary-trace-lem}
There exists a Radon measure $\nu\in M_s(\partial\Omega\times(0,T))$ satisfying
\begin{equation}\label{lateral-limit}
 \lim_{\varepsilon\to 0^+}\int_0^{T_1}\int_{\partial\Omega_{\varepsilon}}u(x,\tau)\tilde{h}(x,\tau)\,d\sigma(x)\,d\tau=\iint_{\partial\Omega\times(0,T_1)}h\,d\nu
\quad\forall 0<T_1<T
\end{equation}
for any bounded continuous function $h$ on $\partial\Omega\times(0,T)$ where $\tilde{h}$ is any bounded continuous extension of $h$ in a tubular neighbourhood of $\partial\Omega\times(0,T)$. Moreover the measure $\nu$ is uniquely given by
\begin{equation}\label{nu-identity}
\iint_{\partial\Omega\times(0,T)}h\,d\nu
=\int_0^T\int_{\Omega}u\left(\overline{h}\Delta \overline{\delta}+2\nabla\overline{h}\cdot \nabla \overline{\delta} +\overline{\delta}\Delta\overline{h}+\overline{\delta}\overline{h}_t\right)\,dx\,dt
\end{equation}
for any $h\in C_c(\partial\Omega\times(0,T))$ where $\overline{h}$ is a continuous extension of $h$ to $\overline{\Omega}\times (0,T)$ vanishing near $t=0, T$, such that  $\overline{h}$ is constant along each inner normal direction of $\partial\Omega$ in $(\overline{\Omega}\setminus\Omega_{\varepsilon_0})\times (0,T)$.
\end{lem}
\begin{proof}
For every bounded continuous function $h$ on $\partial\Omega\times(0,T),$ we extend it to  a bounded continuous function $\overline{h}$ on $(\overline{\Omega}\setminus\Omega_{\varepsilon_0})\times (0,T)$ such that $\overline{h}$ is constant along each inner normal direction of $\partial\Omega$. In view of Lemma \ref{lateral-u-integral-bd-lem}, it suffices to prove the lemma by taking  $\tilde{h}$ to be $\overline{h}$.

First we fix $T_1\in (0,T)$.  For each $\varepsilon\in(0,\varepsilon_0]$, define a linear functional $\Lambda_\varepsilon$ on $C_c(\partial\Omega\times(0,T_1))$ by
$$
\Lambda_\varepsilon h=\int_0^{T_1}\int_{\partial\Omega_{\varepsilon}}\overline{h}(x,\tau)u(x,\tau)\,d\sigma (x)\,d\tau\quad\forall h\in C_c(\partial\Omega\times(0,T_1)).
$$
By Lemma \ref{lateral-u-integral-bd-lem}, we have
\begin{equation*}
|\Lambda_\varepsilon h|\leq C\|h\|_{L^\infty}\quad\forall h\in C_c(\partial\Omega\times(0,T_1)).
\end{equation*}
Therefore by the Riesz representation theorem (Theorem 7.2.8 of \cite{Co}), there exists a Radon measure $\nu_\varepsilon$ on $\partial\Omega\times(0,T_1)$ such that
\begin{align}\label{functional-identity}
&\Lambda_\varepsilon h =\iint_{\partial\Omega\times (0,T)} h\,d\nu_\varepsilon\quad \forall h\in C_c(\partial\Omega\times(0,T_1))\notag\\
\Rightarrow\quad&\int_0^{T_1}\int_{\partial\Omega_{\varepsilon}}\overline{h}(x,\tau)u(x,\tau)\,d\sigma (x)\,d\tau =\iint_{\partial\Omega\times (0,T)} h\,d\nu_\varepsilon
\quad\forall h\in C_c(\partial\Omega\times(0,T_1)).
\end{align}
We now choose a sequence of monotone increasing functions $\{h_i\}_{i=1}^{\infty}\subset C_c(\partial\Omega\times(0,T_1))$, $0\le h_i\le 1$ for any $i\in\Z^+$, satisfying
\begin{equation*}
h_i(x,t)=1\quad\forall x\in\partial\Omega, \tfrac{1}{i}\le t\le T_1-\tfrac{1}{i},i\in\Z^+.
\end{equation*}
Putting $h=h_i$ in \eqref{functional-identity} and letting $i\to\infty$, by Lemma \ref{lateral-u-integral-bd-lem}, we have
\begin{equation}\label{nu-epsilon-measure-bd}
\sup_{0<\varepsilon\le\varepsilon_0}\nu_\varepsilon(\partial\Omega\times (0,T_1))<\infty.
\end{equation}
The measure $\nu_\varepsilon$ depends on $T_1$.  However, by letting $T_1\nearrow T$ it is clear that we could define a Radon measure, still denoted by $\nu_\varepsilon$ in $M_s(\partial\Omega\times(0,T))$ so that \eqref{nu-epsilon-measure-bd} remains valid for any $0<T_1<T$ and
\begin{equation}\label{nu-epsilon-defn}
\iint_{\partial\Omega\times(0,T)}h\,d\nu_\varepsilon=\int_0^T\int_{\partial\Omega_\varepsilon}\overline{h}(x,t)u(x,t)\,d\sigma(x)\, dt\quad\forall h\in C_c(\partial\Omega\times(0,T)).
\end{equation}
By \eqref{nu-epsilon-measure-bd} and weak compactness any sequence $\{\varepsilon_j\}_{j=1}^{\infty}\subset (0,\varepsilon_0)$, $\varepsilon_j\to 0$ as $j\to\infty$, has a subsequence which we may assume without loss of generality to be the sequence itself that converges weakly to $\nu$ as $j\to \infty$ for some $\nu\in M_s(\partial\Omega\times (0,T))$. That is
\begin{equation*}
\lim_{j\to\infty}\iint_{\partial\Omega\times(0,T)}h\,d\nu_{\varepsilon_j}=\iint_{\partial\Omega\times(0,T)}h\,d\nu\quad\forall h\in C_c(\partial\Omega\times(0,T))
\end{equation*}
and
\begin{equation*}
\nu(\partial\Omega\times(0,T_1))\leq \sup_{0<\varepsilon\le\varepsilon_0}\nu_\varepsilon(\partial\Omega\times (0,T_1))<\infty\quad\forall 0<T_1<T.
\end{equation*}

We now claim that the limit $\nu$ is unique. In order to prove this claim
we extend  $\overline{h}$ to a bounded continuous function on $\overline{\Omega}\times (0,T)$ vanishing near $t=0,T$. For any $\varepsilon\in (0,\varepsilon_0)$, let $\eta_\varepsilon(x,t)=(\overline{\delta}(x)-\varepsilon)\overline{h}(x,t)$.  For any $\varepsilon\in (0,\varepsilon_0)$, multiplying  the heat equation for $u$ by $\eta_{\varepsilon}$ and integrating over $\Omega_{\varepsilon}\times (0,T)$, by integration by parts, we have
\begin{eqnarray}\label{u-h1-eqn}
&& \int_0^T\int_{\Omega_\varepsilon} u\left[\overline{h}\Delta \overline{\delta}(x)+
2\nabla\overline{h}\cdot \nabla \overline{\delta}(x) +(\overline{\delta}(x)-\varepsilon)\Delta \overline{h}+(\overline{\delta}(x)-\varepsilon)\overline{h}_t\right]\,dx\,dt\notag\\
&=& -\int_0^T\int_{\partial\Omega_\varepsilon}u(x,t)\overline{h}(x,t)\dfrac{\partial\overline{\delta}}{\partial N}(x)\,d\sigma(x)\,dt\notag\\
&=&\int_0^T\int_{\partial\Omega_\varepsilon}u(x,t)\overline{h}(x,t)\,d\sigma(x)\,dt
\end{eqnarray}
where $\partial/\partial N$ is the derivative with respect to the unit inner normal
$N$ at $\partial\Omega_{\varepsilon}$. Since
\begin{equation*}
\overline{h}\Delta\overline{\delta}(x)=\nabla \overline{h}\cdot \nabla\overline{\delta}(x)=0\quad\forall x\in\Omega\setminus\Omega_{\varepsilon_0},
\end{equation*}
by \eqref{u-h1-eqn}, we have
\begin{align}\label{u-h1-eqn2}
& \int_0^T\int_{\Omega_{\varepsilon_0}}\left( u\overline{h}\Delta \overline{\delta}(x)+
2u\nabla \overline{h}\cdot \nabla \overline{\delta}(x)\right)\,dx\,dt\notag\\
&\qquad +\int_0^T\int_{\Omega_\varepsilon} u\left[(\overline{\delta}(x)-\varepsilon)\Delta \overline{h}+(\overline{\delta}(x)-\varepsilon)\overline{h}_t\right]\,dx\,dt\notag\\
=&\int_0^T\int_{\partial\Omega_\varepsilon}u(x,t)\overline{h}(x,t)\,d\sigma(x)\,dt
\quad\forall 0<\varepsilon<\varepsilon_0.
\end{align}
Putting $\varepsilon=\varepsilon_j$ in \eqref{u-h1-eqn2} and letting $j\to\infty$, by Lemma \ref{lateral-u-integral-bd-lem} and the Lebesgue dominated convergence theorem, we get \eqref{nu-identity}.
Hence the measure $\nu$ is uniquely determined by \eqref{nu-identity} and the  claim holds. Since the sequence $\{\varepsilon_j\}_{j=1}^{\infty}$ is arbitrary, $\nu_{\varepsilon}$ converges weakly to $\nu$ as $\varepsilon\to 0$. This together with  \eqref{nu-epsilon-defn} implies that $\nu$ is the lateral trace of $u$.

Finally, since $\nu$ is finite on $\partial\Omega\times(0,T_1)$ for any $0<T_1<T$, by an approximation argument, one can show that \eqref{lateral-limit} holds not only for any $h\in C_c(\partial\Omega\times(0,T))$ but also for any bounded continuous functions on $\partial\Omega\times(0,T)$.
\end{proof}

\begin{lem}\label{initial-trace-existence-lem}
There exists $(\mu, \lambda)\in M(\Omega,\delta)\times M(\partial\Omega)$ such that for any $\eta\in C_0^\infty(\overline{\Omega})$, we have
\begin{equation}\label{u-initial-value-integral-sense}
\lim_{t\to 0^+}\int_\Omega u(x,t)\eta(x)\,dx = \int_\Omega \eta\,d\mu +\int_{\partial\Omega}\dfrac{\partial \eta}{\partial N}\,d\lambda
\end{equation}
where $\partial/\partial N$ is the derivative with respect to the unit inner normal
$N$ at $\partial\Omega$.
\end{lem}
\begin{proof}
For any $t\in (0,T)$, $0<\varepsilon\le\varepsilon_0$, let
\begin{equation*}
w_\varepsilon(x,t):=G_{\varepsilon}(u(\cdot,t))(x):=\int_{\Omega_{\varepsilon}}G_{\varepsilon}(x,y)u(y,t)\,dy\quad\forall x\in\overline{\Omega}_{\varepsilon}
\end{equation*}
and
\begin{equation}\label{w-defn}
w(x,t):=G(u(\cdot,t))(x):=\int_{\Omega}G(x,y)u(y,t)\,dy\quad\forall x\in\Omega
\end{equation}
be the  Green potential of $u(\cdot,t)$ with respect to the domain $\Omega_{\varepsilon}$ and $\Omega$ respectively where $G(x,y)$, $G_{\varepsilon}(x,y)$, are the Green functions for the Laplacian $-\Delta$ on $\Omega$, $\Omega_{\varepsilon}$, respectively. Then $w_{\varepsilon}\ge 0$ on $\Omega_{\varepsilon}$. By elliptic regularity theory \cite{GT}, for each $0<t<T$, $w_{\varepsilon}$ is a classical solution of
\begin{equation*}
\left\{\begin{aligned}
-\Delta w_{\varepsilon}=&u(\cdot,t)\quad\mbox{ in }\Omega_{\varepsilon}\\
w_{\varepsilon}=&0\qquad\quad\mbox{ on }\partial\Omega_{\varepsilon}
\end{aligned}\right.
\end{equation*}
and $w(\cdot, t)$ is a classical solution of
\begin{equation}\label{w-elliptic-eqn}
-\Delta w(\cdot, t)=u(\cdot, t)\quad\mbox{ in }\Omega.
\end{equation}
Now by Theorem 2.3 of \cite{Wi} for any $x\in\Omega$, there exists a constant $C>0$ such that
\begin{equation}\label{elliptic-green-fcn-estimate}
G(x,y)\le C\delta(y)|x-y|^{1-n}\quad\forall x,y\in\Omega.
\end{equation}
Since for any $x,y\in\Omega$, $G_{\varepsilon}(x,y)$ increases to $G(x,y)$ as $\varepsilon\to 0$, by \eqref{elliptic-green-fcn-estimate}, Lemma \ref{u-weighted-l1-integral-uniform-bd-lem} and the Lebesgue dominated convergence theorem,
$w_{\varepsilon}(x,t)$ increases to $w(x,t)$ as $\varepsilon\to 0$ for any $x\in\Omega$, $0<t<T$.
Hence by \eqref{w-defn}, \eqref{elliptic-green-fcn-estimate} and a direct computation, we get
\begin{equation}\label{w-l1-bd}
\|w(\cdot,t)\|_{L^1(\Omega)}\leq C\int_\Omega u(x,t)\delta(x)\,dx\quad\forall 0<t<T
\end{equation}
for some constant $C>0$.
Now
\begin{equation}\label{w-epsilon-t-derivative}
w_{\varepsilon,t}=G_{\varepsilon}(u_t)=G_{\varepsilon}(\Delta u)=-u(x,t)+\int_{\partial\Omega_{\varepsilon}}\frac{\partial G_{\varepsilon}}{\partial N_y}(x,y)u(y,t)\,d\sigma (y)
\end{equation}
Integrating \eqref{w-epsilon-t-derivative} over $(t_1,t_2)$, $0<t_1<t_2<T_1<T$, we have

\begin{eqnarray*}
& & w_\varepsilon(x,t_2) + \int_{t_2}^{T_1}\int_{\partial\Omega_\varepsilon}\frac{\partial G_{\varepsilon}}{\partial N_y}(x,y)u(y,\tau)\,d\sigma(y)\,d\tau\\
&\leq & w_\varepsilon(x,t_1) + \int_{t_1}^{T_1}\int_{\partial\Omega_\varepsilon}\frac{\partial G_{\varepsilon}}{\partial N_y}(x,y)u(y,\tau)\,d\sigma(y)\,d\tau\ .
\end{eqnarray*}
Letting $\varepsilon\to 0$, we have

\begin{align*}
&w(x,t_2)+ \iint_{\partial\Omega\times(t_2, T_1)}\frac{\partial G}{\partial N_y}(x,y)\,d\nu(y,\tau)\notag\\
 \leq & w(x,t_1)+ \iint_{\partial\Omega\times(t_1, T_1)}\frac{\partial G}{\partial N_y}(x,y)\,d\nu(y,\tau)\quad\forall 0<t_1<t_2<T_1<T.
\end{align*}
Let $0<T_1<T$. This inequality shows that for each $x\in \Omega$,  the function

$$
H(x,t):=w(x,t)+ \iint_{\partial\Omega\times(t,T_1)}\frac{\partial G}{\partial N_y}(x,y)\,d\nu(y,\tau)
$$
is decreasing in $t$.  Hence for each $x\in \Omega$,
$H^{\ast}(x):=\lim_{t\to 0}H(x,t)$ exists.
Since
\begin{equation*}
\iint_{\partial\Omega\times(0,T_1)}\frac{\partial G}{\partial N_y}(x,y)\,d\nu(y,\tau)
\end{equation*}
exists and is finite, we conclude that

$$
w^*(x):=\lim_{t\to 0^+} w(x,t)
$$
exists. Letting $t\to 0$ in \eqref{w-l1-bd}, by Lemma \ref{u-weighted-l1-integral-uniform-bd-lem}, $w^{\ast}\in L^1(\Omega)$.

Sinice by Theorem 2.3 of \cite{Wi} there exists a constant $C>0$ such that
\begin{equation*}
|\nabla_yG(x,y)|\le C|x-y|^{1-n}\quad\forall x,y\in\Omega,
\end{equation*}
we have
\begin{equation*}
\int_{\Omega}\left(\iint_{\partial\Omega\times(0,t)}\frac{\partial G}{\partial N_y}(x,y)\,d\nu(y,\tau)\right)\,dx\le C\nu (\partial\Omega\times (0,t))<\infty\quad\forall 0<t<T.
\end{equation*}
Hence we have
\begin{align*}
&\int_{\Omega}|w^{\ast}(x)-w(x,t)|\,dx\notag\\
\le&\int_{\Omega}\left|H^{\ast}(x)-H(x,t)\right|\,dx
+\int_{\Omega}\left(\iint_{\partial\Omega\times(0,t)}\frac{\partial G}{\partial N_y}(x,y)\,d\nu(y,\tau)\right)\,dx\notag\\
\le&\int_{\Omega}\left|H^{\ast}(x)-H(x,t)\right|\,dx
+C\nu (\partial\Omega\times (0,t)).
\end{align*}
Letting $t\to 0$, we have
\begin{equation}\label{w-to-w-ast-l1}
\lim_{t\to 0}\int_{\Omega}|w^{\ast}(x)-w(x,t)|\,dx=0.
\end{equation}
By \eqref{w-elliptic-eqn} for each $0<t<T$, $w(\cdot,t)$ is superharmonic in $\Omega$. Hence $w^*$ is superharmonic in $\Omega$.

The rest of the proof follows from the arguments based on the proof of theorem 7 in [DK].  First of all, by \eqref{w-defn}, \eqref{w-to-w-ast-l1}, Lemma \ref{u-weighted-l1-integral-uniform-bd-lem} and the Fubini Theorem, for any $\eta\in C_0^{\infty}(\overline{\Omega})$, we have
\begin{align}\label{u-w-integral-eqn}
\int_{\Omega}u(x,t)\eta (x)\,dx=&-\int_{\Omega}\left(\int_{\Omega}G(x,y)\Delta\eta (y)\,dy\right)u(x,t)\,dx\notag\\
=&-\int_{\Omega}\left(\int_{\Omega}G(x,y)u(x,t)\,dx\right)\Delta\eta(y)\,dy\notag\\
=&-\int_{\Omega}w(x,t)\Delta\eta(x)\,dx\notag\\
\to&-\int_{\Omega}w^{\ast}\Delta\eta\,dx\quad\mbox{ as }t\to 0.
\end{align}
On the other hand by the Riesz representation theorem for superharmonic functions \cite{He}, there exists $\mu\in M(\Omega,\delta)$ and a nonnegative harmonic function $h$ in $\Omega$ such that
\begin{equation}\label{w-ast-representation}
w^*(x)=\int_\Omega G(x,y)\,d\mu +h(x)\quad\forall x\in \Omega,
\end{equation}
where $G(x,y)$ is the Green function for the Laplacian $-\Delta$ in $\Omega$.
Applying Martin representation theorem \cite{He} to $h$, there exists $\lambda\in M(\partial\Omega)$ such that
\begin{equation}\label{h-representation}
h(x)=\int_{\partial\Omega}\frac{\partial G}{\partial N_y}(x,y)\,d\lambda(y)\quad\forall x\in \Omega.
\end{equation}
By \eqref{u-w-integral-eqn},\eqref{w-ast-representation}, \eqref{h-representation} and an argument similar to the proof of Theorem 7 of \cite{DK} we get \eqref{u-initial-value-integral-sense} and the lemma follows.

\end{proof}

\section{Existence of the lateral trace and the representation formula}
\setcounter{equation}{0}
\setcounter{theorem}{0}

In this section we will prove Theorem \ref{representation-thm}.  The existence of lateral trace and initial trace have been established in Lemma \ref{existence-bdary-trace-lem} and Lemma \ref{initial-trace-existence-lem} respectively.  It remains to prove \eqref{representation-formula}.

\begin{lem}\label{green-function-derivative-lateral-integral-limit-lem}
For any $x\in\Omega$, $0<s<t<T$, we have
\begin{equation*}
\lim_{\varepsilon\to 0}\int_s^t\int_{\partial\Omega_{\varepsilon}}\dfrac{\partial G_{\varepsilon}}{\partial N_y}(x,t; y,\tau)u(y,\tau)\,d\sigma(y)\,d\tau
=\iint_{\partial\Omega\times (s,t)}\dfrac{\partial G}{\partial N_y}(x,t;y,\tau)\,d\nu(y,\tau).
\end{equation*}
\end{lem}
\begin{proof}
Let $x\in\Omega$ and $0<\delta_1<(t-s)/3$. We choose $0<\varepsilon_1<\min (\varepsilon_0,\delta (x)/2)$ and let $0<\varepsilon<\varepsilon_1$. Then we have
\begin{align}\label{bdary-integral-estimates}
& \left|\int_s^t\int_{\partial\Omega_{\varepsilon}}\dfrac{\partial G_\varepsilon}{\partial N_y}(x,t; y,\tau)u(y,\tau)\,d\sigma(y)\,d\tau
-\iint_{\partial\Omega\times (s,t)}\dfrac{\partial G}{\partial N_y}(x,t; y,\tau)d\nu(y,\tau)\right|\notag\\
 \leq &\int_s^{t-\delta_1}\int_{\partial\Omega_{\varepsilon}}\left|\dfrac{\partial G_{\varepsilon}}{\partial N_y}(x,t;y,\tau)-\frac{\partial G}{\partial N_y}(x,t; y,\tau)\right|u(y,\tau)\,d\sigma(y)\,d\tau\notag\\
+&\left|\int_s^{t-\delta_1}\int_{\partial\Omega_{\varepsilon}}\frac{\partial G}{\partial N_y}(x,t; y,\tau)u(y,\tau)\,d\sigma(y)d\tau-\iint_{\partial\Omega\times(s,t-\delta_1)}\frac{\partial G}{\partial N_y}(x,t; y,\tau)\,d\nu(y,\tau)\right|\notag\\
+&\int_{t-\delta_1}^t\int_{\partial\Omega_{\varepsilon}}\dfrac{\partial G_{\varepsilon}}{\partial N_y}(x,t;y,\tau)u(y,\tau)\,d\sigma(y)d\tau
+\iint_{\partial\Omega\times (t-\delta_1,t)}\dfrac{\partial G}{\partial N_y}(x,t; y,\tau)\,d\nu(y,\tau)\notag\\
\equiv &I_1+I_2+I_3+I_4.
\end{align}
Since by Lemma \ref{lateral-u-integral-bd-lem},
\begin{equation}\label{nu-epsilon-uniform-bd}
\sup_{\varepsilon\in (0,\varepsilon_0)}\nu_\varepsilon(\partial\Omega\times (0,t))<\infty,
\end{equation}
we have
\begin{align}\label{I1-limit}
 I_1\le& \nu_{\varepsilon}(\partial\Omega\times (s,t-\delta_1))\ \underset{\substack{s\le\tau\le t-\delta_1\\y\in\partial\Omega_{\varepsilon}}}{\max}|\nabla_y(G-G_{\varepsilon})(x,t;y,\tau)|\notag\\
 \le & C\ \underset{\substack{s\le\tau\le t-\delta_1\\y\in\partial\Omega_{\varepsilon}}}{\max}|\nabla_y(G-G_{\varepsilon})(x,t;y,\tau)|\notag\\
 \to&0\qquad\mbox{ as }\varepsilon\to 0.
\end{align}
Next,
\begin{align}\label{I2-limit}
I_2
\le&\left|\int_s^{t-\delta_1}\int_{\partial\Omega_{\varepsilon}}\dfrac{\partial G}{\partial N_y}(x,t; y,\tau)u(y,\tau)\,d\sigma(y)\,d\tau\right.\notag\\
-&\left.\int_s^{t-\delta_1}\int_{\partial\Omega_{\varepsilon}}\dfrac{\partial G}{\partial N_{z(y)}}(x,t; z(y),\tau)u(y,\tau)\,d\sigma(y)\,d\tau\right|\notag\\
+&\left|\iint_{\partial\Omega\times(s,t-\delta_1)}\dfrac{\partial G}{\partial N_z}(x,t; z,\tau)\,d\nu_{\varepsilon}(z,\tau)
-\iint_{\partial\Omega\times(s,t-\delta_1)}\dfrac{\partial G}{\partial N_z}(x,t;z,\tau)\,d\nu(z,\tau)\right|\notag\\
\leq &\nu_{\varepsilon}(\partial\Omega\times (s,t-\delta_1))\underset{\substack{s\le\tau\le t-\delta_1\\y\in\partial\Omega_{\varepsilon}}}{\max}\left|\left(\dfrac{\partial G}{\partial N_y}(x,t;y,\tau)-\dfrac{\partial G}{\partial N_{z(y)}}(x,t;z(y),\tau)\right)\right|\notag\\
+&\left|\iint_{\partial\Omega\times(s,t-\delta_1)}\dfrac{\partial G}{\partial N_z}(x,t; z,\tau)\,d\nu_{\varepsilon}(z,\tau)
-\iint_{\partial\Omega\times(s,t-\delta_1)}\dfrac{\partial G}{\partial N_z}(x,t; z,\tau)\,d\nu(z,\tau)\right|\notag\\
\leq& C\underset{\substack{s\le\tau\le t-\delta_1\\y\in\partial\Omega_{\varepsilon}}}{\max}\left|\left(\dfrac{\partial G}{\partial N_y}(x,t; y,\tau)-\dfrac{\partial G}{\partial N_{z(y)}}(x,t; z(y),\tau)\right)\right|\notag\\
+&\left|\iint_{\partial\Omega\times(s,t-\delta_1)}\dfrac{\partial G}{\partial N_z}(x,t;z,\tau)\,d\nu_{\varepsilon}(z,\tau)
-\iint_{\partial\Omega\times(s,t-\delta_1)}\dfrac{\partial G}{\partial N_z}(x,t;z,\tau)\,d\nu(z,\tau)\right|.\notag\\
\to&0\qquad\mbox{ as }\varepsilon\to 0.
\end{align}
Finally, since for  any $y\in\partial\Omega_{\varepsilon}$ we have $|x-y|\ge\delta(x)/2$ which together with the result of [H] implies that

\begin{equation}\label{green-epsilon-derivative-uniform-bd10}
\dfrac{\partial G_\varepsilon}{\partial N_y}(x,t; y,\tau)\leq \dfrac{C_1}{(t-\tau)^{\frac{n+1}{2}}}e^{-\frac{C_2\delta^2(x)}{t-\tau}}\ ,\quad \forall y\in \partial \Omega_\varepsilon, 0<\tau<t\ , 0<\varepsilon<\varepsilon_1
\end{equation}
and
\begin{equation}\label{green-derivative-uniform-bd10}
\dfrac{\partial G}{\partial N_y}(x,t; y,\tau)\leq \dfrac{C_1}{(t-\tau)^{\frac{n+1}{2}}}e^{-\frac{C_2\delta^2(x)}{t-\tau}}\ ,\quad \forall y\in \partial \Omega, 0<\tau<t\ ,
\end{equation}
for some positive constants $C_1$ and $C_2$.  Therefore by \eqref{nu-epsilon-uniform-bd}, \eqref{green-epsilon-derivative-uniform-bd10} and  \eqref{green-derivative-uniform-bd10}, given any small $\varepsilon'>0$, we can choose  $\delta_1$ sufficiently small such that
\begin{align}\label{I34-upper-bd}
 I_3+I_4
\leq &\iint_{\partial\Omega_{\varepsilon}\times(t-\delta_1,t)}
\dfrac{C_1e^{-\frac{C_2\delta^2(x)}{t-\tau}}}{(t-\tau)^{\frac{n+1}{2}}}\,d\nu_\varepsilon(y,\tau)
+\iint_{\partial\Omega\times (t-\delta_1,t)}\dfrac{C_1e^{-\frac{C_2\delta^2(x)}{t-\tau}}}{(t-\tau)^{\frac{n+1}{2}}}\,d\nu(y,\tau)\notag\\
\leq & Ca_0\varepsilon',
\end{align}
for some constant $C>0$ where
\begin{equation*}
a_0=\nu(\partial\Omega\times (0,t))+\sup_{\varepsilon\in (0,\varepsilon_0)}\nu_\varepsilon(\partial\Omega\times (0,t))<\infty.
\end{equation*}
By \eqref{bdary-integral-estimates}, \eqref{I1-limit}, \eqref{I2-limit} and  \eqref{I34-upper-bd}, we have
\begin{align}\label{green-derivative-lateral-integral-limit-bd}
&\limsup_{\varepsilon\to 0}\left|\int_s^t\int_{\partial\Omega_{\varepsilon}}\dfrac{\partial G_{\varepsilon}}{\partial N_y}(x,t;y,\tau)u(y,\tau)\,d\sigma(y)\,d\tau
-\iint_{\partial\Omega\times (s,t)}\dfrac{\partial G}{\partial N_y}(x,t; y,\tau)d\nu(y,\tau)\right|\notag\\
&\leq Ca_0\varepsilon'.
\end{align}
Letting $\varepsilon'\to 0$ in \eqref{green-derivative-lateral-integral-limit-bd} and the lemma follows.
\end{proof}

We are now ready to prove Theorem \ref{representation-thm}.

\begin{proof}[\bf{Proof of Theorem \ref{representation-thm}}:]

\noindent First of all, by letting $\varepsilon\to 0$ in \eqref{representation1}, with the help of Lemma \ref{green-function-derivative-lateral-integral-limit-lem}, we have
for any  $x\in\Omega$ and $0<\tau<t<T$,

\begin{align}\label{u-representaton5}
u(x,t)=&\int_{\Omega}G(x,t;y,\tau)u(y,\tau)\,dy +\iint_{\partial\Omega\times (\tau,t)}\dfrac{\partial G}{\partial N_y}(x,t;y,s)\,d\nu(y,s)\notag\\
=&\int_{\Omega}G(x,t;y,0)u(y,\tau)\,dy +\int_\Omega \left(G(x,t;w,\tau)-G(x,t;w,0)\right)u(w,\tau)\,dw\notag \\
+&\iint_{\partial\Omega\times (\tau,t)}\dfrac{\partial G}{\partial N_y}(x,t;y,s)\,d\nu(y,s)\notag\\
=& J_1+J_2 + J_3.
\end{align}
By Lemma \ref{initial-trace-existence-lem}, we have
\begin{equation}\label{J1-limit}
\lim_{\tau\to 0}J_1=\int_\Omega G(x,t;y,0)\,d\mu(y) +\int_{\partial\Omega}\dfrac{\partial G}{\partial N_y}(x,t;y,0)\,d\lambda(y)\ .
\end{equation}
By the monotone convergence theorem, we have
\begin{equation}\label{J3-limit}
\lim_{\tau\to 0}J_3=\iint_{\partial\Omega\times (0,t)}\dfrac{\partial G}{\partial N_y}(x,t;y,s)\,d\nu(y,s)\ .
\end{equation}
Finally for any $x\in\Omega$, $0<\tau<\tfrac{t}{2}<t<T$, we have
\begin{align}\label{J2-limit}
J_2=&\left|\int_{\Omega} u(w,\tau)\int_0^{\tau}G_s(x,t;w,s)\,ds\,dw\right|\notag \\
\leq &\left|\int_{\overline{\Omega}_{\varepsilon_0}}u(w,\tau)\int_0^{\tau}G_s(x,t;w,s)\,ds\,dw
\right|\notag\\
& +\left|\int_{\Omega\setminus \overline{\Omega}_{\varepsilon_0}}u(w,\tau)\int_0^1\int_0^{\tau}\dfrac{d G_s}{da}(x,t;z(w)+a\delta(w)N_{z(w)},s)\,ds\, da\,dw\right|\notag \\
\leq &C\tau\int_{\overline{\Omega}_{\varepsilon_0}}u(w,\tau)\delta (w)\,dw\notag\\
+&\left|\int_{\Omega}u(w,\tau)\delta (w)\int_0^1\int_0^{\tau}\sum_{j=1}^nN_j(z(w))\dfrac{\partial G_s}{\partial y_j}(x,t;z(w)+a\delta(w)N_{z(w)},s)\,ds\, da\,dw\right|\notag \\
\leq & C\tau\int_{\Omega}u(w,\tau)\delta (w)\,dw\notag \\
\to & 0\quad \text{ as } \tau\to 0
\end{align}
where $N_z=(N_1(z),\cdots,N_n(z))$ is the unit inner normal on $\partial\Omega$  for any $z\in\partial\Omega$.
By \eqref{u-representaton5}, \eqref{J1-limit}, \eqref{J3-limit} and \eqref{J2-limit},
\eqref{representation-formula} follows.

To prove the converse part of the theorem, given $(\mu, \lambda, \nu)\in M(\Omega, \delta)\times M(\partial\Omega)\times M_s(\partial\Omega\times(0,T))$, let $u$ be given by \eqref{representation-formula} and
\begin{equation}\label{representation-formula12}
\left\{\begin{aligned}
v_1(x,t)=&\int_\Omega G(x,t;y,0)\,d\mu(y)+\int_{\partial\Omega}\dfrac{\partial G}{\partial N_y}(x,t;y,0)\,d\lambda(y)\quad\forall x\in\overline{\Omega}, t>0, \notag\\
v_2(x,t)=&\iint_{\partial\Omega\times(0,t)}\dfrac{\partial G}{\partial N_y}(x,t;y,s)\,d\nu(y,s)\quad\forall x\in\Omega, 0<t<T
\end{aligned}\right.
\end{equation}
where $\partial/\partial N_y$ is the derivative with respect to the unit inner normal $N_y$ at $y\in\partial\Omega$.
Then
\begin{equation*}
u(x,t)=v_1(x,t)+v_2(x,t)\quad\forall (x,t)\in Q_T.
\end{equation*}
By the results of \cite{H} and standard parabolic theory (\cite{F}, \cite{LSU}),
$v_1\in C^{2,1}(\overline{\Omega}\times (0,\infty))$ satisfies
\begin{equation*}
\left\{\begin{aligned}
v_{1,t}=&\Delta v_1\quad\mbox{ in }\Omega\times (0,\infty)\\
v_1=&0\qquad\mbox{ on }\partial\Omega\times (0,\infty)
\end{aligned}\right.
\end{equation*}
and
\begin{equation*}
\lim_{t\to 0}\int_{\Omega}v_1\eta\,dx=\int_{\Omega}\eta\, d\mu+\int_{\partial\Omega}\dfrac{\partial\eta}{\partial N}\,d\lambda\quad\forall \eta\in C_0^{\infty}(\overline{\Omega}).
\end{equation*}
Hence $v_1$ has initial traces $(\mu,\lambda)$ and zero lateral trace. Thus it suffices
to prove that $v_2$ satisfies the heat equation in $Q_T$, has initial traces $(0,0)$
and lateral trace $\nu$.

We choose a sequence  $\nu_j$ where $d\nu_j=\varphi_j\,d\sigma\, dt, \varphi_j\in C(\partial\Omega\times[0,T])$, such that
$$
\lim_{j\to\infty}\int_{\partial\Omega\times(0,T_1)}h\,d\nu_j =\int_{\partial\Omega\times(0,T_1)} h\,d\nu\quad \forall h\in C(\partial\Omega\times[0,T_1]), 0<T_1<T.
$$
Let
\begin{equation}
v_{2,j}(x,t)=\iint_{\partial\Omega\times(0,t)}\dfrac{\partial G}{\partial N_y}(x,t;y,s)\varphi_j\,d\sigma ds\quad\forall x\in\Omega, 0<t<T,j\in\Z^+.
\end{equation}
By standard parabolic theory each $v_{2,j}\in C^{2,1}(\overline{\Omega}\times (0,\infty))$ is a classical solution of the heat equation in $Q_T$ with initial traces $(0,0)$ and lateral value $\varphi_j$. Hence  $v_{2,j}$ has $(0,0, \nu_j)$ as its trace triple for any $j\in\Z^+$.  Moreover, for each $(x,t)\in Q_T$, $v_{2,j}(x,t)$ converges to $v_2(x,t)$ as $j\to\infty$ and for each $t\in (0,T)$, $v_{2,j}(\cdot, t)$ converges to $v_2(\cdot, t)$ in $L^1(\Omega)$ as $j\to\infty$. Also  for each $T_1\in (0,T)$ the sequence $\{\|v_{2,j}\|_{L^1(Q_{T_1})}\}_{j=1}^{\infty}$ are uniformly bounded .  Without loss of generality we may assume that $v_{2,j}$ converges weakly to $v_2$ in each $Q_{T_1}$  as $j\to\infty$ for any $0<T_1<T$.  Since $v_{2,j}$ satisfies the heat equation
in $Q_T$, we have
\begin{align*}
&\int_0^T\int_\Omega v_{2,j}(\varphi_t+\Delta \varphi)\,dx\,dt=0\ ,\quad \forall \varphi\in C^\infty_c(Q_T), j\in\Z^+\notag\\
\Rightarrow\quad &\int_0^T\int_\Omega v_2(\varphi_t+\Delta \varphi)\,dx\,dt=0\ ,\quad \forall \varphi\in C^\infty_c(Q_T)\quad\mbox{ as }j\to\infty.
\end{align*}
Hence $v_2$ is a weak solution of the heat equation. Thus by standard by parabolic theory \cite{LSU} $u$ is  a classical solution of the heat equation in $Q_T$.

Next, let $\nu'$ be the lateral trace of $v_2$. Then by Lemma \ref{existence-bdary-trace-lem} for any $h\in C_c(\partial\Omega\times(0,T))$, we have
\begin{equation}\label{nu'-identity}
\iint_{\partial\Omega\times(0,T)}h\,d\nu'
=\int_0^T\int_{\Omega}v_2\left[\overline{h}\Delta \overline{\delta}+2\nabla\overline{h}\cdot \nabla \overline{\delta} +\overline{\delta}\Delta\overline{h}+\overline{\delta}\overline{h}_t\right]\,dx\,dt
\end{equation}
where $\overline{h}$ is a continuous extension of $h$ to $\overline{\Omega}\times (0,T)$ vanishing near $t=0, T$, such that  $\overline{h}$ is constant along each inner normal direction of $\partial\Omega$ in $(\overline{\Omega}\setminus\Omega_{\varepsilon_0})\times (0,T)$.

On the other hand, since the lateral trace of $v_{2,j}$ is $\nu_j$, by putting $u=v_{2,j}$ and $\nu=\nu_j$ in \eqref{nu-identity} and letting $j\to \infty$, we get for any $h\in C_c(\partial\Omega\times(0,T))$, $\nu$ satisfies \eqref{nu-identity}  with $u=v_2$ snd
$\overline{h}$ being a continuous extension of $h$ to $\overline{\Omega}\times (0,T)$ vanishing near $t=0, T$, such that  $\overline{h}$ is constant along each inner normal direction of $\partial\Omega$ in $(\overline{\Omega}\setminus\Omega_{\varepsilon_0})\times (0,T)$. Hence by \eqref{nu-identity} and \eqref{nu'-identity}, $\nu'\equiv\nu$.

Multiplying the heat equation satisfied by $v_{2,j}$ with some $\eta\in C^\infty_0(\overline{\Omega})$, by integration by parts, we have
\begin{align*}
&\int_\Omega v_{2,j}(x,t)\eta(x)\,dx =\int_0^t\int_\Omega v_{2,j}\Delta \eta\,dx\,ds +\int_0^t\int_{\partial\Omega}\varphi_j\dfrac{\partial\eta}{\partial N}\,d\sigma\, dt\notag\\
\Rightarrow\quad&\int_\Omega v_2(x,t)\eta(x)\,dx =\int_0^t\int_\Omega v_2\Delta \eta\,dx\,ds +\iint_{\partial\Omega\times (0,t)}\dfrac{\partial\eta}{\partial N}\,d\nu
\quad\mbox{ as }j\to\infty\notag\\
\Rightarrow\quad&\lim_{t\to 0}\int_\Omega v_2(x,t)\eta(x)\,dx =0.
\end{align*}
Hence the initial trace of $v_2$ is $(0,0)$ and the theorem follows.
\end{proof}


\begin{thebibliography}{99}

\bibitem[A]{A}
D.G.~Aronson, {\em Non-negative solutions of linear parabolic equations}, Annali della Scuola Normale Superiore di Pisa 22 (1968), no. 4, 607--694.

\bibitem[AC]{AC}
D.G.~Aronson and L.A.~Caffarelli, {\em The initial trace of a solution of the porous medium equation}, Trans. Amer. Math. Soc. 280 (1983), 351--366. 

\bibitem[C]{C}
I.~Chavel, {\em Eigenvalues in Riemannian Geometry}, Academic Press, U.S.A., 1984.

\bibitem[CK1]{CK1}
K.S.~Chou and Y.C.~Kwong, {\em The trace triple for nonnegative solutions of generalized porous medium equations}, Calulus of Variation and PDE's 58 (2019), Article number: 23.

\bibitem[CK2]{CK2}
K.S.~Chou and Y.C.~Kwong,  {\em Nonnegative solutions of the porous medium equation with continuous lateral boundary data}, preprint 2019.

\bibitem[Co]{Co}
D.L.~Cohn, {Measure theory}, Second edition, Birkh\" auser advanced texts Basler Lehrbucher, Springer, New York 2013.

\bibitem[DK]{DK}
B.E.J.~Dahlberg and C.E.~Kenig, {\em Non-negative solution of the initial-Dirichlet problem for generalized porous medium equations in cylinders}, J. Amer. Math. Soc. 1 (1988), no. 2, 401--412.

\bibitem[DiH1]{DiH1} E.~DiBenedetto and M.A.~Herrero, {\em On the Cauchy problem and initial traces for a degenerate parabolic equation}, Trans. Amer. Soc. 314 (1989), 187--224.

\bibitem[DiH2]{DiH2} E.~DiBenedetto and M.A.~Herrero, {\em Nonnegative solutions of the evolution $p$-Laplacian equation. Initial traces and Cauchy problem when $1<p<2$}, Arch. Rational Mech. Anal. 111(1990), 225--290.

\bibitem[EG]{EG}
L.C.~Evans and R.F.~Gariepy, {\em Measure theory and fine properties of functions}, Revised edition, CRC Press, Boca Raton, 2015.

\bibitem[FGS]{FGS}
E.B.~Fabes, N.~ Garofalo and S.~Salsa, {\em A backward Harnack inequality and Fatou theorem for nonnegative solutions of parabolic equations}, Illinois J. Math. 30 (1986), no. 4, 536-565.

\bibitem[F]{F}
A.~Friedman, {\em Partial Differential Equations of Parabolic Type}, Dover Publications, Mineola, N.Y., U.S.A., 2008.

\bibitem[GT]{GT}
D.~Gilbarg and N.S.~Trudinger,  {\em Elliptic partial differential equations of second order}, Springer-Verlag, Berlin, Heidelberg 2001.


\bibitem[He]{He}
L.L.~Helms, {\em Introduction to potential theory}, Wiley-Interscience, 1969.

\bibitem[HIT]{HIT}
K.~Hisa, K.~Ishige and J.~Takahashi,  {\em Initial traces and solvability for a semilinear heat equation on a half space of $\mathbb{R}^N$}, Trans. Amer. Math. Soc. electronically published on May 9, 2023, \url{https://doi.org/10.1090/tran/8922}. 

\bibitem[H]{H}
K.M.~Hui,  {\em A Fatou theorem for the solution of the heat equation at the corner points of a cylinder}, Trans. Amer. Math. Soc. 333 (1992), no 2, 607-642.


\bibitem[HW]{HW}
P.~Hartman and A.~Winter, {\em On the solutions of the equation of heat conduction},
Amer. J. Math. 72 (1950), 367--395.

\bibitem[Is]{Is}
K.~Ishige, {\em On the existence of solutions of the Cauchy problem for a doubly nonlinear parabolic equation}, SIAM J. Math. Anal.27 (1996), 1235--1260.

\bibitem[IsK]{IsK}
K.~Ishige and J.~Kinnunen, {\em Intial trace for a doubly nonlinear parabolic equation}, J. Evol. Equ. 11 (2011), 943--957. 

\bibitem[I]{I}
S.~Ito, {\em Diffusion Equations}, Translations of Math. Mono vol. 114, Amer. Math. Soc., Providence, 1992.

\bibitem[K]{K}
J.T.~Kemper,  {\em Temperatures in several variables: kernel functions, representations and parabolic boundary values}, Trans. Amer. Math. Soc., 167 (1972), 243-262.

\bibitem[LSU]{LSU}
O.A. Ladyzenskaya, V.A. Solonnikov and N.N. Uraltceva, {\em Linear and quasilinear equations of parabolic type, } Transl. Math. Mono. vol. 23, Amer. Math. Soc., Providence, R.I., U.S.A., 1968.

\bibitem[MV1]{MV1}
M.~Marcus and L.~V\'eron, {\em Trace aux bord des solutions positives d'\'equations elliptiques et paraboliques non lin\'eaires: r\'esultats d'existence et d'unicit\'e}, C.R. Acad. Sci. Paris 323, Ser. I (1996), 603--608.

\bibitem[MV2]{MV2}
M.~Marcus and L.~V\'eron, {\em The boundary trace of positive solutions of semilinear elliptic equations: the subcritical case}, Arch. Rat. Mech. Anal. 144 (1998), 201--231.

\bibitem[MV3]{MV3} 
M.~Marcus and L.~V\'eron, {\em The boundary trace of positive solutions of semilinear elliptic equations: the supercritical case}, J.~Math.~Pures Appl. 77 (1998), 481--524

\bibitem[MV4]{MV4}
M.~Marcus and L.~V\'eron, {\em Initial trace of positive solutions of some nonlinear parabolic equations}, Commun. Partial Differential Equations 24 (1999), no. 7\&8, 1445-1499.


\bibitem[M]{M}
M.~Murata, {\em Integral representation of nonnegative solutions for parabolic equations and elliptic Martin boundaries}, J. Functional Analysis 245 (2007), 177-212.

\bibitem[W1]{W1}
D.V.~Widder, {\em Positive temperatures on an infintie rod}, Trans. Amer. Math. Soc
55 (1944), 85-95.

\bibitem[W2]{W2}
D.V.~Widder, {\em Positive temperatures on a semi-infinite rod},
Trans. Amer. Math. Soc. 75 (1953), 510--525.

\bibitem[W3]{W3}
D.V.~Widder, {\em The Heat Equation}, Academic Press, NY 1975.

\bibitem[Wi]{Wi}
K.O.~Widman, {\em Inequalities for the Green function and boundary continuity of the gradient of solutions of elliptic differential equations}, Math. Scand. 21 (1968), 17-37.


\end{thebibliography}
\end{document}